\documentclass[12pt]{amsart}
\usepackage{amssymb}
\usepackage[all]{xy}

\newtheorem{theorem}{Theorem}[section]
\newtheorem{definition}[theorem]{Definition}
\newtheorem{proposition}[theorem]{Proposition}

\begin{document}

\title[Noncommutative spheres and projective spaces]{Uniqueness results for noncommutative spheres and projective spaces}

\author{Teodor Banica}
\address{T.B.: Department of Mathematics, Cergy-Pontoise University, 95000 Cergy-Pontoise, France. {\tt teodor.banica@u-cergy.fr}}

\author{Szabolcs M\'esz\'aros}

\address{S.M.: Department of Mathematics, Central European University, 1051 Budapest, Hungary. {\tt szabolcs.thai@gmail.com}}

\subjclass[2000]{46L65 (46L54)}
\keywords{Half-liberation, Noncommutative sphere}

\begin{abstract}
It is known that, under strong combinatorial axioms, $O_N\subset O_N^*\subset O_N^+$ are the only orthogonal quantum groups. We prove here similar results for the noncommutative spheres $S^{N-1}_\mathbb R\subset S^{N-1}_{\mathbb R,*}\subset S^{N-1}_{\mathbb R,+}$, the noncommutative projective spaces $P^{N-1}_\mathbb R\subset P^{N-1}_\mathbb C\subset P^{N-1}_+$, and the projective orthogonal quantum groups $PO_N\subset PO_N^*\subset PO_N^+$.
\end{abstract}

\maketitle

\section*{Introduction}

The concept of half-liberation goes back to \cite{bsp}, \cite{bve}. According to an old result of Brauer \cite{bra}, for the orthogonal group $O_N$, with fundamental representation $u$, we have:
$$Hom(u^{\otimes k},u^{\otimes l})=span\left(T_\pi\Big|\pi\in P_2(k,l)\right)$$

Here $P_2(k,l)$ is the set of pairings between an upper row of $k$ points, and a lower row of $l$ points, and the action of pairings on the tensors over $\mathbb C^N$ is as follows, with the Kronecker symbols $\delta_\pi\in\{0,1\}$ being 1 when all the strings of $\pi$ join pairs of equal indices:
$$T_\pi(e_{i_1}\otimes\ldots\otimes e_{i_k})=\sum_{j_1\ldots j_l}\delta_\pi(^{i_1\ldots i_k}_{j_1\ldots j_l})e_{j_1}\otimes\ldots\otimes e_{j_l}$$

A similar result holds for $O_N^+$. This quantum group, introduced by Wang in \cite{wan}, and satisfying the axioms of Woronowicz in \cite{wo1}, \cite{wo2}, is given by: 
$$C(O_N^+)=C^*\left((u_{ij})_{i,j=1,\ldots,N}\Big|u=\bar{u},u^t=u^{-1}\right)$$

In other words, the passage $O_N\to O_N^+$ is obtained by assuming that the standard coordinates $u_{ij}$ no longer satisfy the commutation relations $ab=ba$. Now since these commutation relations read $T_{\slash\hskip-1.5mm\backslash}\in End(u^{\otimes 2})$, removing them amounts in ``removing the crossings'' from the corresponding set of pairings. We are therefore led to an analogue of the Brauer formula, with $P_2$ being replaced by the set of noncrossing pairings $NC_2$.

This phenomenon, reminding the liberation philosophy in free probability theory \cite{bpa}, \cite{nsp}, \cite{spe}, \cite{vdn}, was investigated in \cite{bsp}, \cite{bve}, one of the findings there being that an intermediate object $O_N^*$ can be inserted, according to the following scheme:
$$\xymatrix@R=7mm@C=10mm{
O_N^+\ar@{.}[r]&\emptyset\ar@{.}[r]&\emptyset\ar@{.}[r]&NC_2\ar[d]\\
O_N^*\ar@{.}[r]\ar[u]&abc=cba\ar@{.}[r]&\slash\hskip-2.0mm\backslash\hskip-1.67mm|\ar@{.}[r]&P_2^*\ar[d]\\
O_N\ar@{.}[r]\ar[u]&ab=ba\ar@{.}[r]&\slash\hskip-2.0mm\backslash\ar@{.}[r]&P_2
}$$

To be more precise, $O_N^*\subset O_N^+$ appears by assuming that the standard coordinates $u_{ij}$ satisfy the ``half-commutation'' relations $abc=cba$. These relations are equivalent to $T_{\slash\hskip-1.5mm\backslash\hskip-1.16mm|}\in End(u^{\otimes 3})$, and the uniqueness result, proved in \cite{bve}, states that the corresponding category $P_2^*=<\slash\hskip-2.0mm\backslash\hskip-1.67mm|>$ is the unique intermediate one $NC_2\subset P\subset P_2$. See \cite{bsp}, \cite{bve}.

We will formulate and prove here a number of similar results, concerning some related geometric objects. We will first discuss the case of noncommutative spheres, with the statemement that, under strong axioms, the spheres $S^{N-1}_\mathbb R\subset S^{N-1}_{\mathbb R,*}\subset S^{N-1}_{\mathbb R,+}$ constructed in \cite{bgo} are the only ones. Then we will discuss the passage from the affine to the projective setting, with uniqueness results both for the associated noncommutative projective spaces $P^{N-1}_\mathbb R\subset P^{N-1}_\mathbb C\subset P^{N-1}_+$, and for their quantum isometry groups $PO_N\subset PO_N^*\subset PO_N^+$.
 
The paper is organized as follows: in \S 1-2 we discuss the spheres and projective spaces, and in \S 3-4 we discuss the associated quantum isometry groups.

\medskip

\noindent {\bf Acknowledgements.} This work was done during the graduate Summer school ``Topological quantum groups'', Bedlewo 2015, and we would like to express our gratitude to the organizers, Uwe Franz, Adam Skalski and Piotr So\l tan. Also, we would like to thank Malte Gerhold, Jan Liszka-Dalecki, and the referee, for several useful suggestions.

\section{Noncommutative spheres}

According to \cite{bgo}, which was inspired from Wang's paper \cite{wan}, and from \cite{bsp}, the free and half-liberated analogues of the unit sphere $S^{N-1}_\mathbb R\subset\mathbb R^N$ are constructed as follows:

\begin{definition}
Associated to any $N\in\mathbb N$ is the following universal $C^*$-algebra:
$$C(S^{N-1}_{\mathbb R,+})=C^*\left(x_1,\ldots,x_N\Big|x_i=x_i^*,\sum_ix_i^2=1\right)$$
The quotient of this algebra by the relations $x_ix_jx_k=x_kx_jx_i$ is denoted $C(S^{N-1}_{\mathbb R,*})$.
\end{definition}

Observe that the above two algebras are indeed well-defined, because the quadratic relations $\sum_ix_i^2=1$ show that we have $||x_i||\leq1$, for any $C^*$-norm. Thus the biggest $C^*$-norm is bounded, and the enveloping $C^*$-algebras are well-defined.

We use the convention that the category of ``noncommutative compact spaces'' is the category of unital $C^*$-algebras, with the arrows reversed. Given such a space $X=Spec(A)$, its classical version $X_{class}$, which is a usual compact space, is the Gelfand spectrum $X_{class}=Spec(A/I)$, where $I\subset A$ is the commutator ideal. We have then:

\begin{proposition}
We have inclusions of noncommutative compact spaces 
$$S^{N-1}_\mathbb R\subset S^{N-1}_{\mathbb R,*}\subset S^{N-1}_{\mathbb R,+}$$ 
and $S^{N-1}_\mathbb R$ is the classical version of both spaces on the right.
\end{proposition}

\begin{proof}
According to the Gelfand and the Stone-Weierstrass theorems, the algebra of continuous functions on the real sphere has the following description:
$$C(S^{N-1}_\mathbb R)=C^*_{comm}\left(x_1,\ldots,x_N\Big|x_i=x_i^*,\sum_ix_i^2=1\right)$$

Thus we have quotient maps $C(S^{N-1}_{\mathbb R,+})\to C(S^{N-1}_{\mathbb R,*})\to C(S^{N-1}_\mathbb R)$, with the second map being obtained by dividing by the commutator ideal, and this gives the result.
\end{proof}

We can axiomatize our spheres, by using the following notion, from \cite{ba1}:

\begin{definition}
A monomial sphere is a subset $S\subset S^{N-1}_{\mathbb R,+}$ obtained via relations of type
$$x_{i_1}\ldots x_{i_k}=x_{i_{\sigma(1)}}\ldots x_{i_{\sigma(k)}},\ \forall (i_1,\ldots,i_k)\in\{1,\ldots,N\}^k$$
with $\sigma\in S_k$ being certain permutations, of variable size $k\in\mathbb N$.
\end{definition}

Equivalently, consider the inductive limit $S_\infty=\bigcup_{k\geq0}S_k$, with the inclusions $S_k\subset S_{k+1}$ being given by $\sigma\in S_k\implies\sigma(k+1)=k+1$. To any $\sigma\in S_\infty$ we can then associate the relations $x_{i_1}\ldots x_{i_k}=x_{i_{\sigma(1)}}\ldots x_{i_{\sigma(k)}}$, for any $(i_1,\ldots,i_k)\in\{1,\ldots,N\}^k$, with $k\in\mathbb N$ being such that $\sigma\in S_k$. Observe that these relations are indeed unchanged when replacing $k\to k+1$, because by using $\sum_ix_i^2=1$ we can always ``simplify'' at right:
\begin{eqnarray*}
x_{i_1}\ldots x_{i_k}x_{i_{k+1}}=x_{i_{\sigma(1)}}\ldots x_{i_{\sigma(k)}}x_{i_{k+1}}
&\implies&x_{i_1}\ldots x_{i_k}x_{i_{k+1}}^2=x_{i_{\sigma(1)}}\ldots x_{i_{\sigma(k)}}x_{i_{k+1}}^2\\
&\implies&\sum_{i_{k+1}}x_{i_1}\ldots x_{i_k}x_{i_{k+1}}^2=\sum_{i_{k+1}}x_{i_{\sigma(1)}}\ldots x_{i_{\sigma(k)}}x_{i_{k+1}}^2\\
&\implies&x_{i_1}\ldots x_{i_k}=x_{i_{\sigma(1)}}\ldots x_{i_{\sigma(k)}}
\end{eqnarray*}

With this convention, a monomial sphere is a subset $S\subset S^{N-1}_{\mathbb R,+}$ obtained via relations $x_{i_1}\ldots x_{i_k}=x_{i_{\sigma(1)}}\ldots x_{i_{\sigma(k)}}$ as above, associated to certain elements $\sigma\in S_\infty$.

Observe that the basic 3 spheres are all monomial, with the permutations producing $S^{N-1}_\mathbb R,S^{N-1}_{\mathbb R,*}$ being the standard crossing and the half-liberated crossing:
$$
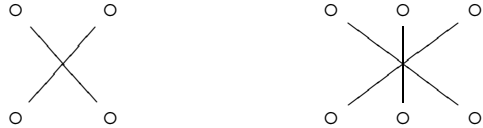
\xymatrix@R=10mm@C=8mm{\circ\ar@{-}[dr]&\circ\ar@{-}[dl]\\\circ&\circ}\qquad\qquad\qquad
\xymatrix@R=10mm@C=5mm{\circ\ar@{-}[drr]&\circ\ar@{-}[d]&\circ\ar@{-}[dll]\\\circ&\circ&\circ}$$ 

Here, and in what follows, we agree to represent the permutations $\sigma\in S_k$ by diagrams between two rows of $k$ points, acting by definition downwards.

With these notions in hand, we can now formulate our main classification result:

\begin{theorem}
The spheres $S^{N-1}_\mathbb R\subset S^{N-1}_{\mathbb R,*}\subset S^{N-1}_{\mathbb R,+}$ are the only monomial ones.
\end{theorem}

\begin{proof}
We follow the approach in \cite{ba1}, where the result was conjectured. We fix a monomial sphere $S\subset S^{N-1}_{\mathbb R,+}$, and we associate to it subsets $G_k\subset S_k$, as follows:
$$G_k=\left\{\sigma\in S_k\Big|x_{i_1}\ldots x_{i_k}=x_{i_{\sigma(1)}}\ldots x_{i_{\sigma(k)}},\forall (i_1,\ldots,i_k)\in\{1,\ldots,N\}^k\right\}$$

Since the relations of type $x_{i_1}\ldots x_{i_k}=x_{i_{\sigma(1)}}\ldots x_{i_{\sigma(k)}}$ can be composed and reversed, each $G_k$ is a group. Moreover, since we have $\sigma\in G_k\implies\sigma|\in G_{k+1}$, we can form the increasing union $G=(G_k)$, which is a subgroup of the increasing union $S_\infty=(S_k)$.

Since the relations $x_{i_1}\ldots x_{i_k}=x_{i_{\sigma(1)}}\ldots x_{i_{\sigma(k)}}$ can be concatenated as well, our group $G=(G_k)$ is ``filtered'', in the sense that it is stable under the operation $(\pi,\sigma)\to\pi\otimes\sigma$. Moreover, $G$ is stable under two more diagrammatic operations, as follows:

(1) Removing outer strings. Indeed, by summing over $a$, we have:
$$Xa=Ya\implies Xa^2=Ya^2\implies X=Y$$
$$aX=aY\implies a^2X=a^2Y\implies X=Y$$

(2) Removing neighboring strings. Indeed, once again by summing over $a$, we have:
$$XabY=ZabT\implies Xa^2Y=Za^2T\implies XY=ZT$$
$$XabY=ZbaT\implies Xa^2Y=Za^2T\implies XY=ZT$$

The problem is that of proving that the only such groups are $\{1\}\subset S_\infty^*\subset S_\infty$, where $S_\infty^*$ is the group associated to the half-liberated sphere $S^{N-1}_{\mathbb R,*}$. So, consider a filtered group $G\subset S_\infty$, assumed non-trivial, $G\neq\{1\}$, and satisfying the above conditions. 

\underline{Step 1.} Our first claim is that $G$ contains a 3-cycle. For this purpose, we use a standard trick, stating that if $\pi,\sigma\in S_\infty$ have support overlapping on exactly one point, say $supp(\pi)\cap supp(\sigma)=\{i\}$, then the commutator $\sigma^{-1}\pi^{-1}\sigma\pi$ is a 3-cycle, namely $(i,\sigma^{-1}(i),\pi^{-1}(i))$. Indeed the computation of the commutator goes as follows:
$$
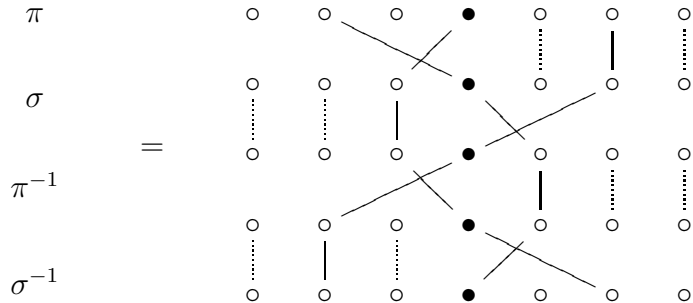
\xymatrix@R=7mm@C=5mm{\pi\\ \sigma\\ \pi^{-1}\\ \sigma^{-1}}\qquad
\xymatrix@R=6mm@C=5mm{\\ \\ =}\qquad
\xymatrix@R=5mm@C=5mm{
\circ&\circ\ar@{-}[drr]&\circ&\bullet\ar@{-}[dl]&\circ\ar@{.}[d]&\circ\ar@{-}[d]&\circ\ar@{.}[d]\\
\circ\ar@{.}[d]&\circ\ar@{.}[d]&\circ\ar@{-}[d]&\bullet\ar@{-}[dr]&\circ&\circ\ar@{-}[dll]&\circ\\
\circ&\circ&\circ\ar@{-}[dr]&\bullet\ar@{-}[dll]&\circ\ar@{-}[d]&\circ\ar@{.}[d]&\circ\ar@{.}[d]\\
\circ\ar@{.}[d]&\circ\ar@{-}[d]&\circ\ar@{.}[d]&\bullet\ar@{-}[drr]&\circ\ar@{-}[dl]&\circ&\circ\\
\circ&\circ&\circ&\bullet&\circ&\circ&\circ
}$$

Now let us pick a non-trivial element $\tau\in G$. By removing outer strings at right and at left we obtain permutations $\tau'\in G_k,\tau''\in G_s$ having a non-trivial action on their right/left leg, and by taking $\pi=\tau'\otimes id_{s-1},\sigma=id_{k-1}\otimes\tau''$, the trick applies.

\underline{Step 2.} Our second claim is $G$ must contain one of the following permutations:
$$
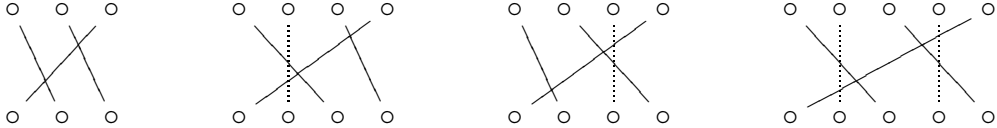
\xymatrix@R=10mm@C=2mm{
\circ\ar@{-}[dr]&\circ\ar@{-}[dr]&\circ\ar@{-}[dll]\\
\circ&\circ&\circ}\qquad\quad
\xymatrix@R=10mm@C=2mm{
\circ\ar@{-}[drr]&\circ\ar@{.}[d]&\circ\ar@{-}[dr]&\circ\ar@{-}[dlll]\\
\circ&\circ&\circ&\circ}\qquad\quad
\xymatrix@R=10mm@C=2mm{
\circ\ar@{-}[dr]&\circ\ar@{-}[drr]&\circ\ar@{.}[d]&\circ\ar@{-}[dlll]\\
\circ&\circ&\circ&\circ}\qquad\quad
\xymatrix@R=10mm@C=2mm{
\circ\ar@{-}[drr]&\circ\ar@{.}[d]&\circ\ar@{-}[drr]&\circ\ar@{.}[d]&\circ\ar@{-}[dllll]\\
\circ&\circ&\circ&\circ&\circ}$$

Indeed, consider the 3-cycle that we just constructed. By removing all outer strings, and then all pairs of adjacent vertical strings, we are left with these permutations.

\underline{Step 3.} Our claim now is that we must have $S_\infty^*\subset G$. Indeed, let us pick one of the permutations that we just constructed, and apply to it our various diagrammatic rules. From the first permutation we can obtain the basic crossing, as follows:
$$
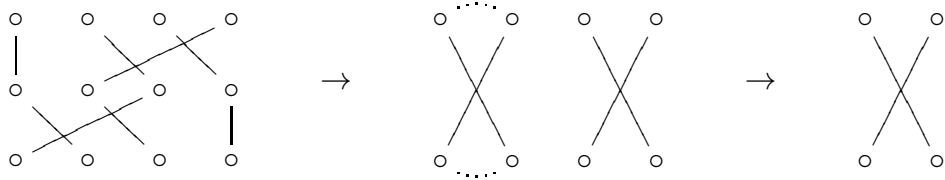
\xymatrix@R=5mm@C=5mm{
\circ\ar@{-}[d]&\circ\ar@{-}[dr]&\circ\ar@{-}[dr]&\circ\ar@{-}[dll]\\
\circ\ar@{-}[dr]&\circ\ar@{-}[dr]&\circ\ar@{-}[dll]&\circ\ar@{-}[d]\\
\circ&\circ&\circ&\circ}\qquad
\xymatrix@R=5mm@C=5mm{
\\ \to\\}\qquad
\xymatrix@R=6mm@C=5mm{
\circ\ar@{-}[ddr]\ar@/^/@{.}[r]&\circ\ar@{-}[ddl]&\circ\ar@{-}[ddr]&\circ\ar@{-}[ddl]\\
\\
\circ\ar@/_/@{.}[r]&\circ&\circ&\circ}\qquad
\xymatrix@R=5mm@C=5mm{
\\ \to\\}\qquad
\xymatrix@R=6mm@C=5mm{
\circ\ar@{-}[ddr]&\circ\ar@{-}[ddl]\\
\\
\circ&\circ}$$

Also, by removing a suitable $\slash\hskip-2.1mm\backslash$ shaped configuration, which is represented by dotted lines in the diagrams below, we can obtain the basic crossing from the second and third permutation, and the half-liberated crossing from the fourth permutation:
$$
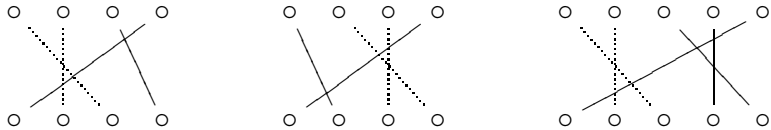
\xymatrix@R=10mm@C=2mm{
\circ\ar@{.}[drr]&\circ\ar@{.}[d]&\circ\ar@{-}[dr]&\circ\ar@{-}[dlll]\\
\circ&\circ&\circ&\circ}\qquad\quad
\xymatrix@R=10mm@C=2mm{
\circ\ar@{-}[dr]&\circ\ar@{.}[drr]&\circ\ar@{.}[d]&\circ\ar@{-}[dlll]\\
\circ&\circ&\circ&\circ}\qquad\quad
\xymatrix@R=10mm@C=2mm{
\circ\ar@{.}[drr]&\circ\ar@{.}[d]&\circ\ar@{-}[drr]&\circ\ar@{-}[d]&\circ\ar@{-}[dllll]\\
\circ&\circ&\circ&\circ&\circ}$$

Thus, in all cases we have a basic or half-liberated crossing, and so $S_\infty^*\subset G$, as claimed.

\underline{Step 4.} Our last claim, which will finish the proof, is that there is no proper intermediate subgroup $S_\infty^*\subset G\subset S_\infty$. In order to prove this, we recall from \cite{ba1} that $S_\infty^*\subset S_\infty$ is the subgroup of  ``parity-preserving'' permutations ($i$ even $\implies$ $\sigma(i)$ even). Equivalently, $S_\infty^*\subset S_\infty$ is the subgroup generated by the transpositions $(1,3)$, $(2,4)$, $(3,5)$, \ldots

Now let us pick an element $\sigma\in S_k-S_k^*$, with $k\in\mathbb N$. We must prove that the group $G=<S_\infty^*,\sigma>$ equals the whole $S_\infty$. In order to do so, we use the fact that $\sigma$ is not parity preserving. Thus, we can find $i$ even such that $\sigma(i)$ is odd. 

In addition, up to passing to $\sigma|$, we can assume that $\sigma(k)=k$, and then, up to passing one more time to $\sigma|$, we can further assume that $k$ is even.

Since both $i,k$ are even we have $(i,k)\in S_k^*$, and so $\sigma(i,k)\sigma^{-1}=(\sigma(i),k)$ belongs to $G$. But, since $\sigma(i)$ is odd, by deleting an appropriate number of vertical strings, $(\sigma(i),k)$ reduces to the basic crossing $(1,2)$. Thus $G=S_\infty$, and we are done.
\end{proof}

Summarizing, we have now a complete axiomatization for the basic 3 spheres. In what follows we will prove some similar results, for some related geometric objects.

\section{Projective spaces}

We discuss here a ``projective version'' of the classification results in section 1. Our starting point is the following functional analytic description of $P^{N-1}_\mathbb R,P^{N-1}_\mathbb C$:

\begin{proposition}
We have presentation results as follows,
\begin{eqnarray*}
C(P^{N-1}_\mathbb R)&=&C^*_{comm}\left((p_{ij})_{i,j=1,\ldots,N}\Big|p=\bar{p}=p^t=p^2,Tr(p)=1\right)\\
C(P^{N-1}_\mathbb C)&=&C^*_{comm}\left((p_{ij})_{i,j=1,\ldots,N}\Big|p=p^*=p^2,Tr(p)=1\right)
\end{eqnarray*}
for the algebras of continuous functions on the real and complex projective spaces.
\end{proposition}

\begin{proof}
This follows indeed from the Gelfand and Stone-Weierstrass theorems, by using the fact that $P^{N-1}_\mathbb R,P^{N-1}_\mathbb C$ are the spaces of rank one projections in $M_N(\mathbb R),M_N(\mathbb C)$.
\end{proof}

The above result suggests the following definition:

\begin{definition}
Associated to any $N\in\mathbb N$ is the following universal algebra,
$$C(P^{N-1}_+)=C^*\left((p_{ij})_{i,j=1,\ldots,N}\Big|p=p^*=p^2,Tr(p)=1\right)$$
whose abstract spectrum is called ``free projective space''.
\end{definition}

Observe that we have embeddings of noncommutative compact spaces $P^{N-1}_\mathbb R\subset P^{N-1}_\mathbb C\subset P^{N-1}_+$, and that the complex projective space $P^{N-1}_\mathbb C$ is the classical version of $P^{N-1}_+$.

Given a closed subset $X\subset S^{N-1}_{\mathbb R,+}$, its projective version is by definition the quotient space $X\to PX$ determined by the fact that $C(PX)\subset C(X)$ is the subalgebra generated by the variables $p_{ij}=x_ix_j$. We have then the following result, from \cite{bgo}:

\begin{proposition}
The projective versions of the $3$ spheres are given by
$$\xymatrix@R=15mm@C=15mm{
S^{N-1}_\mathbb R\ar[r]\ar[d]&S^{N-1}_{\mathbb R,*}\ar[r]\ar[d]&S^{N-1}_{\mathbb R,+}\ar[d]\\
P^{N-1}_\mathbb R\ar[r]&P^{N-1}_\mathbb C\ar[r]&\mathcal P^{N-1}_+}$$
where $\mathcal P^{N-1}_+$ is a certain noncommutative compact space, contained in $P^{N-1}_+$.
\end{proposition}

\begin{proof}
The assertion at left is true by definition. For the assertion at right, we have to prove that the variables $p_{ij}=z_iz_j$ over the free sphere $S^{N-1}_{\mathbb R,+}$ satisfy the defining relations for $C(P^{N-1}_+)$, from Definition 2.2, and the verification here goes as follows:
\begin{eqnarray*}
(p^*)_{ij}&=&p_{ji}^*=(z_jz_i)^*=z_iz_j=p_{ij}\\
(p^2)_{ij}&=&\sum_kp_{ik}p_{kj}=\sum_kz_iz_k^2z_j=z_iz_j=p_{ij}\\
Tr(p)&=&\sum_kp_{kk}=\sum_kz_k^2=1
\end{eqnarray*}

Regarding now the middle assertion, stating that we have $PS^{N-1}_{\mathbb R,*}=P^{N-1}_\mathbb C$, the inclusion ``$\subset$'' follows from the relations $abc=cba$, which imply $abcd=cbad=cbda$. In the other sense now, the point is that we have a matrix model representation, as follows:
$$\pi:C(S^{N-1}_{\mathbb R,*})\to M_2(C(S^{N-1}_\mathbb C))\quad:\quad x_i\to\begin{pmatrix}0&z_i\\ \bar{z}_i&0\end{pmatrix}$$ 

But this gives the missing inclusion ``$\supset$'', and we are done. See \cite{bgo}.
\end{proof}

The inclusion $\mathcal P^{N-1}_+\subset P^{N-1}_+$ is known to be quite similar to an isomorphism, algebrically speaking. To be more precise, when performing the GNS construction with respect to the canonical integration functionals, $\mathcal P^{N-1}_+\subset P^{N-1}_+$ becomes an isomorphism. See \cite{bgo}, \cite{bve}.

We can axiomatize our noncommutative projective spaces, as follows:

\begin{definition}
A monomial space is a subset $P\subset P^{N-1}_+$ obtained via relations of type
$$p_{i_1i_2}\ldots p_{i_{k-1}i_k}=p_{i_{\sigma(1)}i_{\sigma(2)}}\ldots p_{i_{\sigma(k-1)}i_{\sigma(k)}},\ \forall (i_1,\ldots,i_k)\in\{1,\ldots,N\}^k$$
with $\sigma$ ranging over a certain subset of $\bigcup_{k\in2\mathbb N}S_k$, stable under $\sigma\to|\sigma|$.
\end{definition}

Observe the similarity with Definition 1.3. The only subtlety in the projective case is the stability under $\sigma\to|\sigma|$, which in practice means that if the above relation associated to $\sigma$ holds, then the following relation, associated to $|\sigma|$, must hold as well:
$$p_{i_0i_1}\ldots p_{i_ki_{k+1}}=p_{i_0i_{\sigma(1)}}p_{i_{\sigma(2)}i_{\sigma(3)}}\ldots p_{i_{\sigma(k-2)}i_{\sigma(k-1)}}p_{i_{\sigma(k)}i_{k+1}}$$

As an illustration, the basic projective spaces are all monomial:

\begin{proposition}
The $3$ projective spaces are all monomial, with the permutations
$$
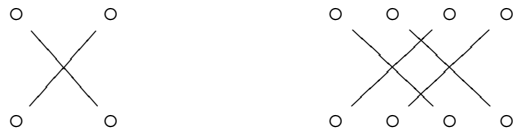
\xymatrix@R=10mm@C=8mm{\circ\ar@{-}[dr]&\circ\ar@{-}[dl]\\\circ&\circ}\qquad\qquad\qquad
\xymatrix@R=10mm@C=3mm{\circ\ar@{-}[drr]&\circ\ar@{-}[drr]&\circ\ar@{-}[dll]&\circ\ar@{-}[dll]\\\circ&\circ&\circ&\circ}$$
producing respectively the spaces $P^{N-1}_\mathbb R,P^{N-1}_\mathbb C$.
\end{proposition}

\begin{proof}
We must divide the algebra $C(P^{N-1}_+)$ by the relations associated to the diagrams in the statement, as well as those associated to their shifted versions, given by:
$$
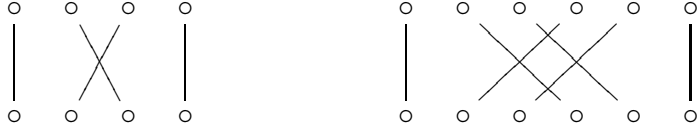
\xymatrix@R=10mm@C=3mm{\circ\ar@{-}[d]&\circ\ar@{-}[dr]&\circ\ar@{-}[dl]&\circ\ar@{-}[d]\\\circ&\circ&\circ&\circ}\qquad\qquad\qquad 
\xymatrix@R=10mm@C=3mm{\circ\ar@{-}[d]&\circ\ar@{-}[drr]&\circ\ar@{-}[drr]&\circ\ar@{-}[dll]&\circ\ar@{-}[dll]&\circ\ar@{-}[d]\\\circ&\circ&\circ&\circ&\circ&\circ}$$ 

(1) The basic crossing, and its shifted version, produce the relations $p_{ab}=p_{ba}$ and $p_{ab}p_{cd}=p_{ac}p_{bd}$. Now by using these relations several times, we obtain:
$$p_{ab}p_{cd}=p_{ac}p_{bd}=p_{ca}p_{db}=p_{cd}p_{ab}$$

Thus, the space produced by the basic crossing is classical, $P\subset P^{N-1}_\mathbb C$, and by using one more time the relations $p_{ab}=p_{ba}$ we conclude that we have $P=P^{N-1}_\mathbb R$, as claimed.

(2) The fattened crossing, and its shifted version, produce the relations $p_{ab}p_{cd}=p_{cd}p_{ab}$ and $p_{ab}p_{cd}p_{ef}=p_{ad}p_{eb}p_{cf}$. The first relations tell us that the projective space must be classical, $P\subset P^{N-1}_\mathbb C$. Now observe that with $p_{ij}=z_i\bar{z}_j$, the second relations read:
$$z_a\bar{z}_bz_c\bar{z}_dz_e\bar{z}_f=z_a\bar{z}_dz_e\bar{z}_bz_c\bar{z}_f$$

Since these relations are automatic, we have $P=P^{N-1}_\mathbb C$, and we are done.
\end{proof}

We can now formulate our projective classification result, as follows:

\begin{theorem}
The projective spaces $P^{N-1}_\mathbb R\subset P^{N-1}_\mathbb C\subset P^{N-1}_+$ are the only monomial ones.
\end{theorem}

\begin{proof}
We follow the proof from the affine case. Let $\mathcal R_\sigma$ be the collection of relations associated to a permutation $\sigma\in S_k$ with $k\in 2\mathbb N$, as in Definition 2.4. We fix a monomial projective space $P\subset P^{N-1}_+$, and we associate to it subsets $G_k\subset S_k$, as follows:
$$G_k=\begin{cases}
\{\sigma\in S_k|\mathcal R_\sigma\ {\rm hold\ over\ }P\}&(k\ {\rm even})\\
\{\sigma\in S_k|\mathcal R_{|\sigma}\ {\rm hold\ over\ }P\}&(k\ {\rm odd})
\end{cases}$$

As in the affine case, we obtain in this way a filtered group $G=(G_k)$, which is stable under removing outer strings, and under removing neighboring strings. Thus the computations in the proof of Theorem 1.4 apply, and show that we have only 3 possible situations, corresponding to the 3 projective spaces in Proposition 2.3 above.
\end{proof}

\section{Quantum isometries}

We discuss now the quantum isometry groups of the spheres and projective spaces. Consider the free orthogonal quantum group $O_N^+$, with standard coordinates denoted $u_{ij}$. Consider as well the subgroup $O_N^*\subset O_N^+$ obtained by assuming that the standard coordinates $u_{ij}$ satisfy the half-commutation relations $abc=cba$. See \cite{bsp}, \cite{bve}.

Given a closed subgroup $G\subset O_N^+$, its projective version $G\to PG$ is by definition given by the fact that $C(PG)\subset C(G)$ is the subalgebra generated by the variables $w_{ij,ab}=u_{ia}u_{jb}$. In the classical case we recover in this way the usual projective version, $PG=G/(G\cap\mathbb Z_2^N)$. It is also known that we have $PO_N^*=PU_N$. See \cite{bve}. 

We use the following action formalism, inspired from \cite{gos}, \cite{wan}:

\begin{definition}
Consider a closed subgroup $G\subset O_N^+$, and a closed subset $X\subset S^{N-1}_{\mathbb R,+}$.
\begin{enumerate}
\item We write $G\curvearrowright X$ when the formula $\Phi(z_i)=\sum_au_{ia}\otimes z_a$ defines a morphism of $C^*$-algebras $\Phi:C(X)\to C(G)\otimes C(X)$.

\item We write $PG\curvearrowright PX$ when the formula $\Phi(z_iz_j)=\sum_au_{ia}u_{jb}\otimes z_az_b$ defines a morphism of $C^*$-algebras $\Phi:C(PX)\to C(PG)\otimes C(PX)$.
\end{enumerate}
\end{definition}

Observe that the above morphisms $\Phi$, if they exist, are automatically coaction maps. Observe also that an affine action $G\curvearrowright X$ produces a projective action $PG\curvearrowright PX$. Finally, let us mention that given an algebraic subset $X\subset S^{N-1}_{\mathbb R,+}$, it is routine to prove that there exist universal quantum groups $G\subset O_N^+$ acting as (1), and as in (2).

We have the following result, with respect to the above notions:

\begin{theorem}
The quantum isometry groups of the spheres and projective spaces are
$$\xymatrix@R=15mm@C=15mm{
O_N\ar[r]\ar[d]&O_N^*\ar[r]\ar[d]&O_N^+\ar[d]\\
PO_N\ar[r]&PU_N\ar[r]&PO_N^+}$$
with respect to the affine and projective action notions introduced above.
\end{theorem}

\begin{proof}
The fact that the 3 quantum groups on top act affinely on the corresponding 3 spheres is known since \cite{bgo}, and is elementary. By restriction, the 3 quantum groups on the bottom follow to act on the corresponding 3 projective spaces.

We must prove now that all these actions are universal. At right there is nothing to prove, so we are left with studying the actions on $S^{N-1}_\mathbb R,S^{N-1}_{\mathbb R,*}$ and on $P^{N-1}_\mathbb R,P^{N-1}_\mathbb C$.

\underline{$S^{N-1}_\mathbb R$.} Here the fact that the action $O_N\curvearrowright S^{N-1}_\mathbb R$ is universal is known from \cite{bhg}, and follows as well from the fact that the action $PO_N\curvearrowright P^{N-1}_\mathbb R$ is universal, proved below.

\underline{$S^{N-1}_{\mathbb R,*}$.} The situation is similar here, with the universality of $O_N^*\curvearrowright S^{N-1}_{\mathbb R,*}$ being proved in \cite{ba2}, and following as well from the universality of $PU_N\curvearrowright P^{N-1}_\mathbb C$, proved below.

\underline{$P^{N-1}_\mathbb R$.} In terms of the projective coordinates $w_{ia,jb}=u_{ia}u_{jb}$ and $p_{ij}=z_iz_j$, the coaction map is given by $\Phi(p_{ij})=\sum_{ab}w_{ia,jb}\otimes p_{ab}$, and we have:
\begin{eqnarray*}
\Phi(p_{ij})&=&\sum_{a<b}(w_{ij,ab}+w_{ij,ba})\otimes p_{ab}+\sum_aw_{ij,aa}\otimes p_{aa}\\
\Phi(p_{ji})&=&\sum_{a<b}(w_{ji,ab}+w_{ji,ba})\otimes p_{ab}+\sum_aw_{ji,aa}\otimes p_{aa}
\end{eqnarray*}

By comparing these two formulae, and then by using the linear independence of the variables $p_{ab}=z_az_b$ for $a\leq b$, we conclude that we must have:
$$w_{ij,ab}+w_{ij,ba}=w_{ji,ab}+w_{ji,ba}$$

Let us apply now the antipode to this formula. For this purpose, observe first that we have $S(w_{ij,ab})=S(u_{ia}u_{jb})=S(u_{jb})S(u_{ia})=u_{bj}u_{ai}=w_{ba,ji}$. Thus by applying the antipode we obtain $w_{ba,ji}+w_{ab,ji}=w_{ba,ij}+w_{ab,ij}$, and by relabelling, we obtain: 
$$w_{ji,ba}+w_{ij,ba}=w_{ji,ab}+w_{ij,ab}$$

Now by comparing with the original relation, we obtain $w_{ij,ab}=w_{ji,ba}$. But, with $w_{ij,ab}=u_{ia}u_{jb}$, this formula reads $u_{ia}u_{jb}=u_{jb}u_{ia}$. Thus our quantum group $G\subset O_N^+$ must be classical, $G\subset O_N$, and so we have $PG\subset PO_N$, as claimed.

\underline{$P^{N-1}_\mathbb C$.} Consider a coaction map, written as $\Phi(p_{ij})=\sum_{ab}u_{ia}u_{jb}\otimes p_{ab}$, with $p_{ab}=z_a\bar{z}_b$. The idea here will be that of using the formula $p_{ab}p_{cd}=p_{ad}p_{cb}$. We have:
\begin{eqnarray*}
\Phi(p_{ij}p_{kl})&=&\sum_{abcd}u_{ia}u_{jb}u_{kc}u_{ld}\otimes p_{ab}p_{cd}\\
\Phi(p_{il}p_{kj})&=&\sum_{abcd}u_{ia}u_{ld}u_{kc}u_{jb}\otimes p_{ad}p_{cb}
\end{eqnarray*}

The terms at left being equal, and the last terms at right being equal too, we deduce that, with $[a,b,c]=abc-cba$, we must have the following formula:
$$\sum_{abcd}u_{ia}[u_{jb},u_{kc},u_{ld}]\otimes p_{ab}p_{cd}=0$$

Now since the quantities $p_{ab}p_{cd}=z_a\bar{z}_bz_c\bar{z}_d$ at right depend only on the numbers $|\{a,c\}|,|\{b,d\}|\in\{1,2\}$, and this dependence produces the only possible linear relations between the variables $p_{ab}p_{cd}$, we are led to $2\times2=4$ equations, as follows:

(1) $u_{ia}[u_{jb},u_{ka},u_{lb}]=0$, $\forall a,b$.

(2) $u_{ia}[u_{jb},u_{ka},u_{ld}]+u_{ia}[u_{jd},u_{ka},u_{lb}]=0$, $\forall a$, $\forall b\neq d$.

(3) $u_{ia}[u_{jb},u_{kc},u_{lb}]+u_{ic}[u_{jb},u_{ka},u_{lb}]=0$, $\forall a\neq c$, $\forall b$.

(4) $u_{ia}[u_{jb},u_{kc},u_{ld}]+u_{ia}[u_{jd},u_{kc},u_{lb}]+u_{ic}[u_{jb},u_{ka},u_{ld}]+u_{ic}[u_{jd},u_{ka},u_{lb}]=0$, $\forall a\neq c$, $\forall b\neq d$.

We will need in fact only the first two formulae. Since (1) corresponds to (2) at $b=d$, we conclude that (1,2) are equivalent to (2), with no restriction on the indices. By multiplying now this formula to the left by $u_{ia}$, and then summing over $i$, we obtain:
$$[u_{jb},u_{ka},u_{ld}]+[u_{jd},u_{ka},u_{lb}]=0$$

We use now the antipode/relabel trick from \cite{bhg}. By applying the antipode we obtain $[u_{dl},u_{ak},u_{bj}]+[u_{bl},u_{ak},u_{dj}]=0$, and by relabelling we obtain:
$$[u_{ld},u_{ka},u_{jb}]+[u_{jd},u_{ka},u_{lb}]=0$$

Now by comparing with the original relation, we obtain $[u_{jb},u_{ka},u_{ld}]=[u_{jd},u_{ka},u_{lb}]=0$. Thus our quantum group is half-classical, $G\subset O_N^*$, and we are done.
\end{proof}

The above results can be probably further improved. As an example here, let us say that a closed subgroup $G\subset U_N^+$ acts projectively on $PX$ when we have a coaction map $\Phi(z_iz_j)=\sum_{ab}u_{ia}u_{jb}^*\otimes z_az_b$. Then the above proof can be adapted, by putting $*$ signs where needed, and so Theorem 3.2 still holds, under this more general formalism. However, establishing the most general universality results, involving arbitrary subgroups $H\subset PO_N^+$, looks like a quite non-trivial question, and we have no results here.

\section{Projective easiness}

We discuss here the analogues of the classification results in sections 1-2, for the quantum groups introduced in section 3. First, we have the following key result, from \cite{bbc}:

\begin{proposition}
We have the following results:
\begin{enumerate}
\item The quantum group inclusion $O_N\subset O_N^*$ is maximal.

\item The group inclusion $PO_N\subset PU_N$ is maximal.
\end{enumerate}
\end{proposition}

\begin{proof}
The idea here is that (2) can be obtained by using standard Lie group tricks, and (1) follows then from it, via standard algebraic lifting results. See \cite{bbc}.
\end{proof}

Our claim now is that, under suitable assumptions, $O_N^*$ is the only intermediate object $O_N\subset G\subset O_N^+$, and $PU_N$ is the only intermediate object $PO_N\subset G\subset PO_N^+$. In order to formulate a precise statement here, we recall the following notion, from \cite{bsp}:

\begin{definition}
An intermediate quantum group $O_N\subset G\subset O_N^+$ is called easy when 
$$span(NC_2(k,l))\subset Hom(u^{\otimes k},u^{\otimes l})\subset span(P_2(k,l))$$
comes via $Hom(u^{\otimes k},u^{\otimes l})=span(D(k,l))$, for certain sets of pairings $D(k,l)$.
\end{definition}

As explained in \cite{bsp}, by ``saturating'' the sets $D(k,l)$, we can assume that the collection $D=(D(k,l))$ is a category of pairings, in the sense that it is stable under vertical and horizontal concatenation, upside-down turning, and contains the semicircle. See \cite{bsp}.

In the projective case now, we have the following related definition:

\begin{definition}
A projective category of pairings is a collection of subsets 
$$NC_2(2k,2l)\subset E(k,l)\subset P_2(2k,2l)$$
stable under the usual categorical operations, and satisfying $\sigma\in E\implies |\sigma|\in E$.
\end{definition}

As basic examples here, we have the categories $NC_2\subset P_2^*\subset P_2$, where $P_2^*$ is the category of matching pairings. This follows indeed from definitions.

Now with the above notion in hand, we can formulate:

\begin{definition}
A quantum group $PO_N\subset H\subset PO_N^+$ is called projectively easy when 
$$span(NC_2(2k,2l))\subset Hom(v^{\otimes k},v^{\otimes l})\subset span(P_2(2k,2l))$$
comes via $Hom(v^{\otimes k},v^{\otimes l})=span(E(k,l))$, for a certain projective category $E=(E(k,l))$.
\end{definition}

Observe that, given any easy quantum group $O_N\subset G\subset O_N^+$, its projective version $PO_N\subset PG\subset PO_N^+$ is projectively easy in our sense. In particular the quantum groups $PO_N\subset PU_N\subset PO_N^+$ are all projectively easy, coming from $NC_2\subset P_2^*\subset P_2$.

We have in fact the following general result:

\begin{proposition}
We have a bijective correspondence between the affine and projective categories of partitions, given by $G\to PG$ at the quantum group level.
\end{proposition}

\begin{proof}
The construction of correspondence $D\to E$ is clear, simply by setting:
$$E(k,l)=D(2k,2l)$$

Conversely, given $E=(E(k,l))$ as in Definition 4.3, we can set:
$$D(k,l)=\begin{cases}
E(k,l)&(k,l\ {\rm even})\\
\{\sigma:|\sigma\in E(k+1,l+1)\}&(k,l\ {\rm odd})
\end{cases}$$

Our claim is that $D=(D(k,l))$ is a category of partitions. Indeed:

(1) The composition action is clear. Indeed, when looking at the numbers of legs involved, in the even case this is clear, and in the odd case, this follows from:
$$|\sigma,|\sigma'\in E\implies |^\sigma_\tau\in E\implies{\ }^\sigma_\tau\in D$$

(2) For the tensor product axiom, we have 4 cases to be investigated. The even/even case is clear, and the odd/even, even/odd, odd/odd cases follow respectively from:
$$|\sigma,\tau\in E\implies|\sigma\tau\in E\implies\sigma\tau\in D$$
$$\sigma,|\tau\in E\implies|\sigma|,|\tau\in E\implies|\sigma||\tau\in E\implies|\sigma\tau\in E\implies\sigma\tau\in D$$
$$|\sigma,|\tau\in E\implies||\sigma|,|\tau\in E\implies||\sigma||\tau\in E\implies \sigma\tau\in E\implies\sigma\tau\in D$$

(3) Finally, the conjugation axiom is clear from definitions.

Now with these definitions in hand, both compositions $D\to E\to D$ and $E\to D\to E$ follow to be the identities, and the quantum group assertion is clear as well.
\end{proof}

Now back to the uniqueness issues, we have here:

\begin{theorem}
We have the following results:
\begin{enumerate}
\item $O_N^*$ is the only intermediate easy quantum group $O_N\subset G\subset O_N^+$.

\item $PU_N$ is the only intermediate projectively easy quantum group $PO_N\subset G\subset PO_N^+$.
\end{enumerate}
\end{theorem}

\begin{proof}
The assertion regarding $O_N\subset O_N^*\subset O_N^+$ is from \cite{bve}, and the assertion regarding $PO_N\subset PU_N\subset PO_N^+$ follows from it, and from the duality in Proposition 4.5.
\end{proof}

There are of course a number of finer results waiting to be established, regarding the inclusions $S^{N-1}_\mathbb R\subset S^{N-1}_{\mathbb R,*}$ and $P^{N-1}_\mathbb R\subset P^{N-1}_\mathbb C$, with inspiration from Proposition 4.1. The results in \cite{bic}, \cite{bdu} provide in principle useful tools in dealing with such questions.

\end{document}